\documentclass{article}
\usepackage[cp1251]{inputenc}
\usepackage{amssymb}
\usepackage{amsmath}
\usepackage{latexsym}
\usepackage{amsthm}
\newtheorem {theorem}{Theorem}
\newtheorem {definition}{Definition}
\newtheorem {proposition}{Proposition}

\newtheorem {corollary}{Corollary}
\theoremstyle{remark}
\newtheorem{remark}[equation]{Remark}

\DeclareMathOperator{\vol}{vol}
\DeclareMathOperator{\area}{Area}
\begin{document}
\title{Non-maximality of known extremal metrics on torus and Klein bottle}
\author{Mikhail A. Karpukhin}
\date{} 
\maketitle
\begin{abstract}
El Soufi-Ilias' theorem establishes a connection between minimal submanifolds of spheres and extremal metrics for eigenvalues of the Laplace-Beltrami operator. Recently, this connection was used to provide several explicit examples of extremal metrics. We investigate the maximality of these metrics and prove that all of them are not maximal.
\end{abstract}
\section*{Introduction}
Let $M$ be a closed surface and $g$ be a Riemannian metric on $M$. Then the Laplace-Beltrami operator $\Delta$ acts on the space of smooth functions on $M$ by the formula
$$
\Delta f = -\frac{1}{\sqrt{|g|}}\frac{\partial}{\partial x^i}\bigl(\sqrt{|g|}g^{ij}\frac{\partial f}{\partial x^j}\bigr).
$$
It is known that the spectrum of $\Delta$ is discrete and consists only of eigenvalues. Moreover, the multiplicity of any eigenvalue is finite and the sequence of eigenvalues tends to infinity. Let us denote this sequence by 
$$
0 = \lambda_0(M,g) < \lambda_1(M,g) \leqslant \lambda_2(M,g) \leqslant \lambda_3(M,g) \leqslant \ldots,
$$
where the eigenvalues are written with their multiplicities.

For a fixed $M$  the following quantities can be considered as functionals on the space of all Riemannian metrics on $M$,
$$
\Lambda_i(M,g) = \lambda_i(M,g) \area(M,g).
$$
Several recent papers~\cite{EGJ, ElSoufiIlias1, ElSoufiIlias2, Hersh, JNP, Korevaar, LiYau, Nadirashvili1, Nadirashvili2, YangYau} deal with finding supremum of these functionals in the space of all Riemannian metrics on $M$. 

An upper bound for $\Lambda_1(M,g)$ in terms of genus of $M$ was provided in the paper~\cite{YangYau} and the existence of such a bound for $\Lambda_i(M,g)$ was shown in the paper~\cite{Korevaar}.
The exact upper bounds are known for a limited number of functionals: $\Lambda_1(\mathbb{S}^2,g)$ (see~\cite{Hersh}), $\Lambda_1(\mathbb{RP}^2,g)$ (see~\cite{LiYau}), $\Lambda_1(\mathbb{T}^2,g)$ (see~\cite{Nadirashvili1}), $\Lambda_1(\mathbb{K}\mathrm{l},g)$ (see~\cite{EGJ, JNP}), $\Lambda_2(\mathbb{S}^2,g)$ (see~\cite{Nadirashvili2}). We refer to the introduction to the paper~\cite{PenskoiOtsuki} for more details.  

The functional $\Lambda_i(M,g)$ depends continously on $g$ but this functional is not differentiable. However, it is known that for an analytic deformation $g_t$ of the initial metric $g$ there exist the left and right derivatives of $\Lambda_i(M,g_t)$ with respect to $t$, see e.g. the papers~\cite{ElSoufiIlias2,BU,Berger}. This is a motivation for the following definition.

\begin{definition}[see~\cite{ElSoufiIlias1,Nadirashvili1}] A Riemannian metric $g$ on a closed surface $M$ is called an {\em extremal} metric for a functional 
$\Lambda_i(M,g)$ if for any analytic deformation $g_t$ such that $g_0 = g$ the following inequality holds,
$$
\frac{d}{dt}\Lambda_i(M,g_t)\Bigl|_{t=0+} \leqslant 0 \leqslant \frac{d}{dt}\Lambda_i(M,g_t)\Bigl|_{t=0-}.
$$
\end{definition}
\begin{definition}
A metric $g$ is called a {\em maximal} metric for a functional $\Lambda_i(M,g)$ if for any metric $h$ on $M$
$$
\Lambda_i(M,g)\geqslant\Lambda_i(M,h).
$$
\end{definition}
A question whether there exists smooth maximal metric is not trivial. For example there is no smooth maximal metric for $\Lambda_2(\mathbb{S}^2,g)$ (see~\cite{Nadirashvili2}).

The list of known extremal metrics is longer than the list of known exact upper bounds for $\Lambda_i(M,g)$, but until now their maximality has not been studied. In the present paper we investigate the maximality of the known extremal metrics. The list of currently known extremal metrics follows.
\begin{itemize}
\item[(A)] Metrics on the Otsuki tori $O_{p/q}$ were studied in the paper~\cite{PenskoiOtsuki}.
\item[(B)] Metrics on the Lawson tori and Klein bottles $\tau_{m,k}$ were studied in the paper~\cite{PenskoiLawson}.
\item[(C)] Metrics on the surfaces $\tilde\tau_{m,k}$ bipolar to Lawson surfaces were studied in the paper~\cite{Lapointe}.
\item[(D)] Metrics on the bipolar surfaces $\tilde O_{p/q}$ to Otsuki tori were studied in the paper~\cite{Karpukhin}.
\end{itemize}
In further description a Klein bottle is denoted by $\mathbb{K}$.

The definitions of these surfaces are given in the following sections. The main result of the present paper is the following theorem.
\begin{theorem}
There are no maximal metrics among the metrics (A)-(D) except for $\tilde\tau_{3,1}$.
\label{MainTheorem}
\end{theorem}
\begin{remark}
The metric on the Lawson bipolar Klein bottle $\tilde\tau_{3,1}$ is maximal for the functional $\Lambda_1(\mathbb{K},g)$, see~\cite{EGJ,JNP}. 
\end{remark}
We also prove the following proposition.
\begin{proposition}
The metric on the Clifford torus is extremal for an infinite number of functionals $\Lambda_i(M,g)$, but it is not maximal for any of them.
\label{CliffordTheorem}
\end{proposition}  
The extremality of the Clifford torus for an infinite number of functionals $\Lambda_i(M,g)$ is known, but to the best of author's knowledge has not yet been published. In the present paper we fill this gap. 

In the following description 
we use the notations $K(k)$, $E(k)$ and $\Pi(n,k)$ for the elliptic integrals of the first, second and third kind respectively, see~\cite{Friedman},
\begin{equation*}
\begin{split}
&K(k) = \int\limits_0^1\frac{1}{\sqrt{1-x^2}\sqrt{1-k^2x^2}}\,dx,\qquad
E(k) = \int\limits_0^1\frac{\sqrt{1-k^2x^2}}{\sqrt{1-x^2}}\,dx, \\
&\Pi(n,k) = \int\limits_0^1\frac{1}{(1-nx^2)\sqrt{1-x^2}\sqrt{1-k^2x^2}}\,dx. 
\end{split}
\end{equation*}

The paper is organized in the following way. In Section~\ref{bound} we prove lower bounds for $\sup\Lambda_n(\mathbb{T}^2,g)$ and $\sup\Lambda_n(\mathbb{K},g)$. These bounds are used throughout the paper in order to prove the non-maximality of metrics (A)-(D). In Section~\ref{Connection} we recall a connection between extremal metrics and minimal submanifolds of the unit sphere. Section~\ref{OtsukiDef} contains a discription of Otsuki tori as an $SO(2)$-invariant minimal submanifolds of $\mathbb{S}^3$ of cohomogeneity 1. Sections~\ref{OtsukiEstimate}, \ref{LawsonEstimate}, \ref{BipLawsonEstimate}, \ref{BipOtsukiEstimate}
are dedicated to estimates for extremal metrics (A)-(D) respectively and this finishes the proof of Theorem~\ref{MainTheorem}. Finally, Section~\ref{CliffordEstimate} contains the proof of Proposition~\ref{CliffordTheorem}.
\section{Lower bounds for $\sup\Lambda_n$}
The aim of this section is to prove the following proposition (compare with Corollary 4 in the paper~\cite{CE}).
\label{bound}
\begin{proposition}
One has the following inequalities,
$$
\sup\Lambda_n(\mathbb{T},g)\geqslant 8\pi\left(n-1+\frac{\pi}{\sqrt{3}}\right),
$$
$$
\sup\Lambda_n(\mathbb{K},g)\geqslant 8\pi(n-1)+12\pi E\left(\frac{2\sqrt{2}}{3}\right),
$$
where $E(k)$ stands for the elliptic integral of the second kind.
\label{LowerBound}
\end{proposition}
\subsection{Attaching handles due to Chavel-Feldman} Let $M$ be a compact smooth Riemannian manifold of dimension $n\geqslant 2$. Let us pick two distinct points $p_1,p_2\in M$. For $\varepsilon > 0$ we define
\begin{itemize}
\item[$B_\varepsilon$] $\colon =$ union of open geodesic balls of radius $\varepsilon$ about $p_1$ and $p_2$,
\item[$\Omega_\varepsilon$] $\colon = M\backslash B_\varepsilon$,
\item[$\Gamma_\varepsilon$] $\colon = \partial B_\varepsilon = \partial \Omega_\varepsilon$.
\end{itemize}
Here the number $\varepsilon$ is chosen to be less than $\frac{1}{4}$ of injectivity radius of $M$ and less than $\frac{1}{4}$ of a distance between $p_1$ and $p_2$ if $p_1$ and $p_2$ lie in the same connected component of $M$. 
We say that manifold $M_\varepsilon$ is obtained from $M$ by adding a handle across $\Gamma_\varepsilon$ if 
\begin{itemize}
\item[1)] $\Omega_\varepsilon$ is isometrically embedded in $M_\varepsilon$;
\item[2)] there exists a diffeomorphism $\Psi_\varepsilon\colon M_\varepsilon\backslash\Omega_{2\varepsilon}\to [-1,1]\times \mathbb{S}^{n-1}$  such that 
$$
M_\varepsilon\backslash\Omega_\varepsilon = \Psi_\varepsilon^{-1}\left(\left[-\frac{1}{2},\frac{1}{2}\right]\times\mathbb{S}^{n-1}\right).
$$
\end{itemize}
Let us denote by $\lambda_j$ and $\lambda_j(\varepsilon)$ the Laplace spectrum of $M$ and $M_\varepsilon$ respectively.
Chavel and Feldman in their paper~\cite{ChFeld} obtained a sufficient condition for convergence $\lambda_j(\varepsilon)\to\lambda_j$ as $\varepsilon$ tends to $0$. In order to formulate this condition we need to give the following definition.
\begin{definition}
For any compact connected Riemannian manifold $X$ of dimension $n\geqslant 2$, the isoperimetric constant $c_1(X)$ is defined by 
$$
c_1(X)=\inf\limits_Y\frac{(\vol_{n-1}(Y))^n}{(\min(\vol_n(X_1),\vol_n(X_2)))^{n-1}},
$$
where $\vol_k$ stands for $k$-dimensional Riemannian measure, and $Y$ ranges over all compact $(n-1)$-dimensional submanifolds of $X$ such that they divide $X$ into $2$ open 
submanifolds $X_1$, $X_2$ each having boundary $Y$.
\end{definition} 
\begin{theorem}[Chavel, Feldman~\cite{ChFeld}] Assume that $M_\varepsilon$ is connected for any $\varepsilon$ and there exists a constant $c>0$ such that $c_1(M_\varepsilon)\geqslant c$ for all $\varepsilon>0$. Then $\lim_{\varepsilon\to 0}\lambda_j(\varepsilon) = \lambda_j$ for all $j=1,2,\ldots$
\label{Ch1}
\end{theorem}
\begin{remark}
The assumption in the previous theorem implies $\lim_{\varepsilon\to 0}\vol_n(M_\varepsilon)=\vol_n(M)$ by picking $Y=\Gamma_\varepsilon$.
\end{remark}
In the same paper existence of such $M_\varepsilon$ is verified for any surface $M$ and almost any pair of points $p_1$, $p_2$.
\begin{theorem}
Let $M$ be a compact $2$-dimensional Riemannian manifold, $K$ be its Gaussian curvature and $\tilde M = (M\backslash K^{-1}(0))\cup \mathrm{int}\, K^{-1}(0)$. Then $\tilde M$ is open and dense in $M$. Suppose that $p_1,p_2\in\tilde M$ and one of the following possibilities occur:
\begin{itemize}
\item $M$ is connected,
\item $M$ has two connected components and $p_i$ lie in different connected components.
\end{itemize}
Then $M_\varepsilon$ can be constructed so that assumption of Theorem~\ref{Ch1} holds. In particular, $\area(M_\varepsilon)\to \area(M)$ as $\varepsilon\to 0$.
\label{Ch2}
\end{theorem}
\begin{remark} Let us remark that Chavel and Feldman considered only the case of a connected manifold $M$. However, their arguments could be extended almost without changes to the non-connected case as stated above.
\end{remark}
\subsection{Proof of Proposition~\ref{LowerBound}.}
Consider the flat equilateral torus $\tau_{eq}$. After suitable rescaling of the metric we have $\area(\tau_{eq}) = 4\pi^2/\sqrt{3}$ and $\lambda_1(\tau_{eq}) = 2$. For the euclidean sphere $\mathbb{S}^2$ of volume $4\pi$ one also has $\lambda_1(\tau_{eq}) = 2$. Let us take $n-1$ copies of $\mathbb{S}^2$ denoted by $S_i$, $i = 1,2,\ldots,n-1$. 
Thus for $T_n = \tau_{eq}\coprod_{i=1}^{n-1}S_i$ we have $\lambda_n(T_n) = 2$ and therefore $\Lambda_n(T_n) = 8\pi\left(n-1+\pi/\sqrt{3}\right)$.
Consecutive application of Theorem~\ref{Ch2} yields the existance of the sequence $M_\varepsilon$, diffeomorphic to torus, such that $\Lambda_n(M_\varepsilon)\to\Lambda_n(T_n)$ as $\varepsilon$ tends to $0$. This observation completes the proof of the first inequality.

The second inequality can be proved in the same fashion. The only difference is that instead of $\tau_{eq}$ one has to use Lawson bipolar Klein bottle $\tilde\tau_{3,1}$ (see Section~\ref{BipLawsonEstimate} for a defintion). It was proven in the paper~\cite{JNP} that $\Lambda_1(\tilde\tau_{3,1}) = 12\pi E\left(2\sqrt{2}/3\right)$. By a suitable rescaling of the metric on $\tilde\tau_{3,1}$, one can assume that $\lambda_1(\tilde\tau_{3,1}) = 2$ and then apply construction of the previous paragraph. 
\section{Otsuki tori}
\subsection{Connection with minimal submanifolds of the sphere.}
Let $\psi\colon  M \looparrowright \mathbb{S}^n$ be a minimal immersion in the unit sphere with canonical metric $g_{can}$. We denote by $\Delta$ the Laplace-Beltrami operator on $M$ associated with the metric $\psi^*g_{can}$. Let us introduce the Weyl's eigenvalues counting funcion
$$
N(\lambda) = \#\{i|\lambda_i(M,g)<\lambda\}.
$$
The following theorem provides a general approach to finding smooth extremal metrics.
\begin{theorem}[El Soufi, Ilias,~\cite{ElSoufiIlias2}] Let $\psi\colon M \looparrowright 
\mathbb{S}^n$ be a minimal
immersion in the unit sphere $\mathbb{S}^n$ endowed with the canonical metric $g_{can}$. 

Then the metric $\psi^*g_{can}$ on $M$ is extremal for the functional $\Lambda_{N(2)}(M,g)$.
\label{th1}
\end{theorem}

Therefore, one can start with minimal submanifold $N$ of the unit sphere then compute $N(2)$ and the metric induced on $N$ by this immersion is extremal for the functional $\Lambda_{N(2)}(N,g)$. However, for a given minimal submanifold there is no algorithm for computing the exact value of $N(2)$. Nevertheless, this approach was succesfully realized by Penskoi in the papers~\cite{PenskoiOtsuki, PenskoiLawson} for metrics (A),(B) as well as by the author in the paper~\cite{Karpukhin} for metrics (D). Some ideas of this approach was partially used in the paper~\cite{Lapointe} for metrics (C).

\subsection{Reduction theorem for minimal submanifolds.}
 Let $M$ be a \label{Connection} 
Riemannian manifold equipped with a metric $g'$ and let $G$ be a compact group acting on $M$ by isometries. For every point $x\in M$ let us denote by 
$G_x$ the stability subgroup of $x$.
\begin{definition}
For two points $x,y\in M$ we say that $x\preccurlyeq y$ if $G_x\subset gG_yg^{-1}$ for some $g\in G$. The orbit $Gx$ is the {\em orbit of principal type} if for any point $y\in 
M$ one has $x\preccurlyeq y$. 
\end{definition}

Let $M^*$ stand for the union of all orbits of principal type, then $M^*$ is an open dense submanifold of $M$ (see~\cite{MSY}). Moreover, $M^*/G$ carries a natural Riemannian metric $g$ defined by the formula $g(X,Y) = g'(X',Y')$, where $X,Y$ are
tangent vectors at $x\in M^*/G$ and $X',Y'$ are tangent vectors at a point $x'\in\pi^{-1}(x)\subset M^*$ such that $X'$
and $Y'$ are orthogonal to the orbit $\pi^{-1}(x)$ and $d\pi(X')=X,\,d\pi(Y')=Y$.

Let $f\colon N \looparrowright M$ be a $G$-invariant immersed submanifold, i.e. a manifold equipped with 
an action of $G$ by isometries such that $g\cdot f(x) = f(g\cdot x)$ for any $x\in N$.
\begin{definition} A {\em cohomogeneity} of a $G$-invariant immersed submanifold $N$ is the number $\dim N - 
\nu$, where $\nu$ is the dimension of the orbits of principal 
type.
\end{definition}
Let us define for $x\in M^*/G$ a volume function $V(x)$ by the formula $V(x) = \mathrm{Vol}(\pi^{-1}(x))$. Also for each integer
$k \geqslant 1$ let us define a metric $g_k = V^{\frac{2}{k}}g$.
\begin{proposition}[Hsiang, Lawson~\cite{HsiangLawson}] Let $f\colon N\looparrowright M^*$ be a $G$-invariant immersed submanifold of cohomogeneity $k$, and
let $M^*/G$ be equipped with the metric $g_k$. Then $f\colon N\looparrowright M^*$ is minimal if and only if 
$\bar f\colon N/G\looparrowright M^*/G$ is minimal.
\label{HsiangTh}    
\end{proposition}

\subsection{Otsuki tori.} 
Otsuki tori
\label{OtsukiDef} 
were introduced by Otsuki in the paper~\cite{Otsuki}. Let us 
recall the concise description by Penskoi from the paper~\cite{PenskoiOtsuki}. For more details see Section 1.2 of the paper~\cite{PenskoiOtsuki}. Consider the action of $SO(2)$ on 
the three-dimensional unit sphere $\mathbb{S}^3 \subset \mathbb{R}^4$ given by the formula
$$
\alpha\cdot(x,y,z,t) = (\cos\alpha x+\sin\alpha y, -\sin\alpha x +\cos\alpha y, z, t),
$$ 
where $\alpha \in [0,2\pi)$ is a coordinate on $SO(2)$. The space of orbits $\mathbb{S}^3/SO(2)$ is the closed half-sphere
$\mathbb{S}^2_+$,
$$
q^2+z^2+t^2=1, \qquad q\geqslant 0,
$$
where a point $(q,z,t)$ corresponds to the orbit $(q\cos\alpha,q\sin\alpha,z,t) \in \mathbb{S}^3$. The space of principal orbits
$(\mathbb{S}^3)^*/SO(2)$ is the open half sphere 
$$
\mathbb{S}^2_{>0} = \{(q,z,t)\in\mathbb{S}^2|q>0\}.
$$ 
Let us introduce the spherical coordinates in the space of orbits,
$$
\left\{
   \begin{array}{rcl}
	t &=& \cos\varphi\sin\theta,\\
	z &=& \cos\varphi\cos\theta,\\
	q &=& \sin\varphi.\\
   \end{array}
\right.
$$
Since we look for minimal submanifolds of cohomogeneity 1, the Hsiang-Lawson's metric 
is given by the formula
\begin{equation}
V^2(d\varphi^2 + \cos^2\varphi d\theta^2) = 4\pi^2\sin^2\varphi(d\varphi^2 + \cos^2\varphi d\theta^2). \label{OtsukiMetric}
\end{equation}
\begin{definition} An immersed minimal $SO(2)$-invariant two-dimensional torus in $\mathbb{S}^3$ such that its image
 by the projection $\pi\colon\mathbb{S}^3\to\mathbb{S}^3/SO(2)$ is a closed geodesics in $(\mathbb{S}^3)^*/SO(2)$ 
endowed with the metric (\ref{OtsukiMetric}) is called an {\em Otsuki torus}.
\end{definition}
The following proposition was proved in the paper~\cite{PenskoiOtsuki}.
\begin{proposition} Except one particular case given by the equation $\psi = \pi/4$, Otsuki tori are in\label{PenskoiProp} 
one-to-one correspondence with rational numbers $p/q$ such that
$$
\frac{1}{2}<\frac{p}{q}<\frac{\sqrt{2}}{2},\qquad p,q>0,\,(p,q) = 1.
$$
\end{proposition}
\begin{definition} By $O_{p/q}$ we denote the Otsuki torus corresponding to 
$p/q$. Following the paper~\cite{PenskoiOtsuki} we reserve the term "Otsuki tori"
for the tori $O_{p/q}$.
\end{definition} 
In order to fix notations we give a sketch of the proof of Proposition~\ref{PenskoiProp}. 
\begin{proof} Let us use the standard notation for the coefficients of the metric~(\ref{OtsukiMetric}),
$$
E = 4\pi^2\sin^2\varphi, \qquad G = 4\pi^2\sin^2\varphi\cos^2\varphi.
$$
As we know the velocity vector of a geodesic has a constant length. Suppose this length equals 1. Then this assumption as well as the equation of geodesics for 
$\ddot\theta$ provides the following two equations,
\begin{equation}
\label{dtheta}
\dot\theta = \frac{\sin a\cos a}{2\pi\cos^2\varphi\sin^2\varphi},
\end{equation}
\begin{equation}
\label{dvarphi}
\dot\varphi^2 = \frac{\sin^2\varphi\cos^2\varphi - \sin^2 a\cos^2 a}{4\pi^2\sin^4\varphi\cos^2\varphi},
\end{equation} 
where $a$ is the minimal value of $\varphi$ on the geodesic. Then the geodesic is situated in the annulus $a\leqslant \varphi \leqslant \dfrac{\pi}{2} - a$.  
We choose a natural parameter $t$ such that $\varphi(0) = a$.

Let us denote by $\Omega(a)$ the difference between the value of $\theta$ corresponding to $\varphi = a$ and the closest to it value 
of $\theta$ corresponding to $\varphi = \pi/2-a$. It is clear that
$$
\Omega(a) = \sin a\cos a\int\limits_a^{\pi/2-a}\frac{d\varphi}{\cos\varphi\sqrt{\sin^2\varphi\cos^2\varphi - \sin^2 a\cos^2 a}}.
$$

The geodesic is closed iff $\Omega(a) = p\pi/q$. The rest of the proof follows from the following properties of the function
 $\Omega(a)$, see the paper~\cite{Otsuki},
\begin{itemize}
\item[1)] $\Omega(a)$ is continuous and monotonous on $\left(0,\pi/4\right]$,
\item[2)] $\lim_{a\to 0+}\Omega(a) = \pi/2$ and $\Omega\left(\pi/4\right) = \pi/\sqrt{2}$.
\end{itemize}
\end{proof}

\subsection{Estimates for $\Lambda_{2p-1}(O_{p/q})$.} According to the paper~\cite{PenskoiOtsuki}, the metric on an Otsuki torus $O_{p/q}$ is extremal 
\label{OtsukiEstimate}
for the functional $\Lambda_{2p-1}(\mathbb{T}^2,g)$. The goal of this section is to prove the following proposition.
\begin{proposition} For $p,q$, such that $(p,q)=1$ and $1/2<p/q<\sqrt{2}/2$, the following inequality holds, 
$$
8\pi\left(2p-2 +\frac{\pi}{\sqrt{3}} \right)>\Lambda_{2p-1}(O_{p/q}).
$$
\label{OtsukiEstimateProp} 
\end{proposition}

In order to prove Proposition~\ref{OtsukiEstimateProp} we have to prove several auxiliary propositions.
\begin{proposition} For $a\in\left(0,\pi/4\right)$ such that $\Omega(a) = p\pi/q$ one has
$$
\Lambda_{2p-1}(O_{p/q}) = 8\pi q\cos aE\left(\sqrt{1-\tan^2a}\right).
$$
\label{valueOtsuki}
\end{proposition}
\begin{proof} Let us use the notations of Proposition~\ref{PenskoiProp}. As we know, 
$$
\dot \varphi = \pm\frac{\sqrt{G-c^2}}{\sqrt{EG}},
$$
where $c = 2\pi\sin a\cos a$. Therefore, the length of the segment on the geodesic $\pi(O_{p/q})$ between the closest points with $\varphi = a$ and $\varphi=\pi/2-a$ is equal to $2\pi I$, where
$$
I = \int_a^{\pi/2-a}\frac{\sin\varphi}{\sqrt{1-\sin^2a\cos^2a/(\sin^2\varphi\cos^2\varphi)}}d\varphi.
$$

Let us express $I$ in terms of elliptic integrals,
\begin{equation*}
\begin{split}
I =& 
\int_{\sin a}^{\cos a}\frac{x\sqrt{1-x^2}}{\sqrt{x^2(1-x^2) - \cos^2a\sin^2a}}dx = 
\frac{1}{2}\int_{\sin^2a}^{\cos^2a}\frac{\sqrt{1-u}}{\sqrt{u(1-u)-\cos^2a\sin^2a}} \\= 
&\frac{1}{2}\int_0^1\frac{\sqrt{(1-\sin^2a) - (\cos^2a-\sin^2a)t}}{\sqrt{t(1-t)}}dt = 
\frac{1}{2}\cos a\int_0^1\frac{\sqrt{1-(1-\tan^2a)t}}{\sqrt{t(1-t)}}dt 
\\=& \cos a\int_0^1\frac{\sqrt{1-(1-\tan^2a)y^2}}{\sqrt{1-y^2}} = \cos aE(\sqrt{1-\tan^2a}).
\end{split}
\end{equation*}

Here the following changes of variables were used, 
$$
\cos\varphi = x,\quad x^2=u, \quad u=(\cos^2a-\sin^2a)t+\sin^2a,\quad t=y^2.
$$ 
Since the maps $\theta\mapsto\theta+\theta_0$ and $\theta\mapsto\theta_0-\theta$ are isometries, the length of the 
geodesic $\pi(O_{p/q})$ is equal to $4\pi q\cos aE(\sqrt{1-\tan^2a})$. By Proposition 13 from the 
paper~\cite{PenskoiOtsuki}, $\Lambda_{2p-1}(O_{p/q})$ is equal to the doubled length of the geodesic $\pi(O_{p/q})$.
\end{proof} 
\begin{proposition}
For $k\in[0,1]$ one has the following inequality,
$$
K(k) - \frac{2}{2-k^2}E(k)\geqslant 0.
$$
\label{MainLemma}
\end{proposition}
\begin{proof}
Let us expand the left hand side using the definitions of $E$ and $K$,
\begin{equation*}
\begin{split}
K(k) - \frac{2}{2-k^2}E(k) = &\int_0^{\pi/2}\frac{d\theta}{\sqrt{1-k^2\sin^2\theta}} - 
\frac{2}{2-k^2}\int_0^{\pi/2}\sqrt{1-k^2\sin^2\theta}\, d\theta \\=&\frac{k^2}{2-k^2}\int
_0^{\pi/2}\frac{2\sin^2\theta - 1}{\sqrt{1-k^2\sin^2\theta}}d\theta.
\end{split}
\end{equation*}

Since the integrand is negative on $(0,\pi/4)$ and positive on $(\pi/4,\pi/2)$, one has
\begin{equation*}
\begin{split}
\int_0^{\pi/2}&\frac{2\sin^2\theta-1}{\sqrt{1-k^2\sin^2\theta}}d\theta = \int_0^{\pi/4}
\frac{2\sin^2\theta-1}{\sqrt{1-k^2\sin^2\theta}}d\theta + \int_{\pi/4}^{\pi/2}\frac{2\sin^2\theta-1}
{\sqrt{1-k^2\sin^2\theta}}d\theta\\ &\geqslant
\int_0^{\pi/4}\frac{2\sin^2\theta-1}{\sqrt{1-k^2/2}} + \int_{\pi/4}^{\pi/2}
\frac{2\sin^2\theta-1}{\sqrt{1-k^2/2}} \\&= -\frac{1}{\sqrt{1-k^2/2}}\int_0^{\pi/2}\cos 2\theta d\theta = 0.
\end{split}
\end{equation*}
\end{proof}

Let us introduce the notation
$$
\Phi(a) = \cos aE\left(\sqrt{1-\tan^2a}\right).
$$
\begin{proposition} The function $\Phi(a)$ is non-decreasing and $\Phi'(a)<1/2$ for any $a\in \left(0,\dfrac{\pi}{4}\right)$. 
In particular, $1 = \Phi(0) \leqslant \Phi(a) \leqslant \Phi\left(\pi/4\right) = \pi/(2\sqrt{2})$.
\label{PhiProp}
\end{proposition}
\begin{corollary} One has
\begin{equation}
\label{inOtsuki}
4\sqrt{2}\pi^2 q\geqslant\Lambda_{2p-1}(O_{p/q})\geqslant 8\pi q.
\end{equation}
\label{ineqOtsuki}
\end{corollary}
\textbf{Remark.} Let us remark that during the preparation of the manuscript inequality~(\ref{inOtsuki}) appeared in the paper~\cite{HuSong}. 
\begin{proof}[Proof of Proposition~\ref{PhiProp}.]
Let us recall the following formulae for the derivatives of elliptic integrals,
\begin{equation}
\label{dedk}
\frac{dE(k)}{dk} = \frac{E(k)-K(k)}{k}\,,\qquad \frac{dK(k)}{dk} = \frac{E(k)}{k(1-k^2)} - \frac{K(k)}{k},
\end{equation}
\begin{equation}
\label{dpi}
\begin{split}
\frac{\partial\Pi(n,k)}{\partial n} = \frac{1}{2(k^2-n)(n-1)}&\left(E(k) +\frac{(k^2-n)}{n}K(k) + \frac{(n^2-k^2)}{n}\Pi(n,k)\right), \\
\frac{\partial\Pi(n,k)}{\partial k} = &\frac{k}{n-k^2}\left(\frac{E(k)}{k^2-1} + \Pi(n,k)\right).
\end{split}
\end{equation}
 
Let us introduce a notation $\beta = \sqrt{1-\tan^2a}$. One obtains
\begin{equation}
\label{dphi}
\begin{split}
\Phi'(a) &= \cos a\left(-2\tan a \frac{E(\beta)-K(\beta)}{2\cos^2a(1-\tan^2a)}\right) - \sin a E(\beta)\\ &= 
-\sin a\left(E(\beta) + \frac{E(\beta)-K(\beta)}{\cos^2a - \sin^2a}\right)\\ &= \frac{\sqrt{(1-\beta^2)(2-\beta^2)}}{\beta^2}\left(K(\beta) - \frac{2}{2-\beta^2}E(\beta)\right).
\end{split}
\end{equation} 
Now the monotonicity of the function $\Phi(a)$ follows from Proposition~\ref{MainLemma}.

For the proof of the second part, let us go back to formula~(\ref{dphi}). One has
\begin{equation*}
\begin{split}
\Phi'(a) =& -\sin a\left(\frac{2\cos^2aE(\beta)-K(\beta)}{\cos^2a-\sin^2a}\right) \\=& -\frac{\sin a}{\cos^2a-\sin^2a}
\int_0^{\pi/2}\frac{2\cos^2a(1-\beta^2\sin^2\theta)-1}{\sqrt{1-\beta^2\sin^2\theta}}d\theta\\ =& 
\sin a\int_0^{\pi/2}\frac{2\sin^2\theta-1}{\sqrt{1-\beta^2\sin^2\theta}}d\theta\leqslant
\sin a\int_{\pi/4}^{\pi/2}\frac{2\sin^2\theta-1}{\sqrt{1-\beta^2}}d\theta\\ =& 
-\cos a \int_{\pi/4}^{\pi/2}\cos 2\theta d\theta = \cos a\frac{\sin 2\theta}{2}\Bigl|_{\pi/2}
^{\pi/4}\leqslant \frac{1}{2}.
\end{split}
\end{equation*}

This finishes the proof of Proposition~\ref{PhiProp}.
\end{proof}
\begin{proposition}
The function $(2/\pi)\Omega(a) - \Phi(a)$ is increasing on the interval $\left(0,\pi/4\right)$.
\end{proposition}
\begin{proof}
In the paper~\cite{HuSong} the following formula was proved,
$$
\Omega(a) = \frac{1}{\sin a}\Pi\left(-\frac{\cos 2a}{\sin^2a},\sqrt{1-\tan^2a}\right).
$$
Using formulae~(\ref{dpi}) one obtains the following formula,
$$
\frac{d\Omega(a)}{da} = \frac{1}{\cos a\cos 2a}K\left(\sqrt{1-\tan^2a}\right) - \frac{2\cos a}{\cos 2a}E\left(\sqrt{1-\tan^2a}\right).
$$

Let us recall the notation $\beta(a) = \sqrt{1 - \tan^2a}$. Then  one has
\begin{equation}
\label{domega}
\Omega'(a) = \frac{(2-\beta^2)^{\frac{3}{2}}}{\beta^2}\left(K(\beta) - \frac{2}{2-\beta^2}E(\beta)\right),
\end{equation}
$$
\Omega(a) = \sqrt{\frac{2-\beta^2}{1-\beta^2}}\Pi\left(-\frac{\beta^2}{1-\beta^2},\beta\right).
$$

Moreover, by formula~(\ref{dphi}) one has
$$
\Phi'(a) = \frac{\sqrt{(1-\beta^2)(2-\beta^2)}}{\beta^2}\left(K(\beta) - \frac{2}{2-\beta^2}E(\beta)\right).
$$
The inequality $(2/\pi)(2-\beta^2) - \sqrt{1-\beta^2}>0$ and Proposition~\ref{MainLemma} imply the inequality
$$
\frac{2}{\pi}\Omega'(a) - \Phi'(a) = \frac{\sqrt{2-\beta^2}}{k^2}\left(K(\beta) - \frac{2}{2-\beta^2}E(\beta)\right)\left(\frac{2}{\pi}(2-\beta^2) - \sqrt{1-\beta^2}\right) > 0.
$$
\end{proof}
\begin{corollary}
For $a\in\left[1/5,\pi/4\right]$ one has
$$
\frac{2}{\pi}\Omega(a) - \Phi(a) > \frac{2\sqrt{3}-\pi}{3\sqrt{3}}.
$$
\label{cor1}
\end{corollary}
\begin{proof}
Using the tables of elliptic integrals, e.g. the book~\cite{Friedman}, one obtains the inequality
$$
\frac{2}{\pi}\Omega\left(\frac{1}{5}\right) - \Phi\left(\frac{1}{5}\right) > \frac{2\sqrt{3}-\pi}{3\sqrt{3}}.
$$ 
The rest of the proof follows from the monotonicity of the function on the left hand side.
\end{proof}
\begin{proposition} For $\xi\in\left[0,1/5\right]$ one has
$$
\Omega'(\xi) > \frac{\pi}{4}\left(\frac{\pi}{\sqrt{3}}-1\right)^{-1}
$$
\label{prop1}
\end{proposition}
\begin{proof}
By formula~(\ref{domega}) for $\xi\in\left[0,1/5\right]$ one has
\begin{equation*}
\begin{split}
\Omega'(\xi) &= \frac{(2-\beta(\xi)^2)^{\frac{3}{2}}}{\beta(\xi)^2}\left(K(\beta(\xi)) - \frac{2}{2-\beta(\xi)^2}E(\beta(\xi))\right)\\ &\geqslant K\left(\beta\left(\frac{1}{5}\right)\right) - 2\frac{2-\beta^2\left(1/5\right)}{\beta^2\left(1/5\right)}E\left(\beta\left(\frac{1}{5}\right)\right).
\end{split}
\end{equation*}
In the last inequality we used the facts that $K(k)$ is increasing function and $E(k)$ as well as $\beta(a)$ are decreasing functions. The table of the elliptic integrals in the book~\cite{Friedman} provides the inequality
$$
K\left(\beta\left(\frac{1}{5}\right)\right) - 2\frac{2-\beta^2\left(1/5\right)}{\beta^2\left(1/5\right)}E\left(\beta\left(\frac{1}{5}\right)\right)>\frac{\pi}{4}\left(\frac{\pi}{\sqrt{3}}-1\right)^{-1}
$$
which completes the proof.
\end{proof}
\begin{proof}[Proof of Proposition~\ref{OtsukiEstimateProp}] We want to prove that 
$$
8\pi\left(2p-2 + \frac{\pi}{\sqrt{3}}\right) > 8\pi q\Phi(a),
$$
where $\Omega(a) = p\pi/q$. This inequality is equivalent to the following one
$$
2\frac{p}{q} - \frac{2\sqrt{3}-\pi}{q\sqrt{3}} > \Phi(a).
$$
Since $\Omega(a) = p\pi/q$, it is sufficient to prove that 
\begin{equation}
\frac{2}{\pi}\Omega(a) - \Phi(a) > \frac{2\sqrt{3}-\pi}{q\sqrt{3}}.
\label{mainineq}
\end{equation}
Since $q\geqslant 3$, the application of Corollary~\ref{cor1} provides inequality~(\ref{mainineq}) for $a\in\left[1/5,\pi/4\right]$.
In order to prove inequality for $a\in\left[0,1/5\right]$ let us note that by Proposition~\ref{PhiProp} 
\begin{equation*}
\begin{split}
\frac{2}{\pi}\Omega(a) - \Phi(a) &= \frac{2}{\pi}(\Omega(a)-\Omega(0)) - (\Phi(a)-\Phi(0))\\ &= a\left(\frac{2}{\pi}\Omega'(\xi)-\Phi'(\eta)\right)\geqslant 
a\left(\frac{2}{\pi}\Omega'(\xi)-\frac{1}{2}\right)
\end{split}
\end{equation*}
for some $\xi,\eta\in(0,a)$. Moreover,
$$
\frac{1}{2q}\pi\leqslant\frac{2p-q}{2q}\pi = \frac{p}{q}\pi - \frac{1}{2}\pi = \Omega(a)-\Omega(0) = a\Omega'(\xi)
$$
or
$$
\frac{1}{q}<\frac{2a}{\pi}\Omega'(\xi).
$$
Therefore, inequality~(\ref{mainineq}) follows from the inequality
$$
\frac{2}{\pi}\Omega'(\xi) - \frac{1}{2} > \frac{2}{\pi}\left(2-\frac{\pi}{\sqrt{3}}\right)\Omega'(\xi)
$$
or the inequality
$$
\Omega'(\xi) > \frac{\pi}{4}\left(\frac{\pi}{\sqrt{3}}-1\right)^{-1}.
$$
The last inequality easily follows from Proposition~\ref{prop1}. 
\end{proof}
\section{Lawson surfaces} 
A Lawson tau-surface $\tau_{m,k}$ is an immersed surface in the sphere $\mathbb{S}^3$
\label{LawsonEstimate}
defined by the double-periodic immersion of $\mathbb{R}^2$ given by the formula
$$
(\cos mx \cos y,\sin mx\cos y, \cos kx\sin y,\sin kx\sin y).
$$

It was introduced by Lawson in the paper~\cite{Lawson}. He also proved that for each pair $\{m,k\}$, such that 
$m\geqslant k\geqslant 1$ and $(m,k)=1$, the surface $\tau_{m,k}$ is a distinct compact minimal surface in $\mathbb{S}^3$.
Let us assume that $(m,k)=1$ then if both $m$ and $k$ are odd then $\tau_{m,k}$ is a torus, we call it a Lawson torus. Otherwise $\tau_{m,k}$ is a Klein bottle, we call it a Lawson Klein bottle.
\begin{proposition}[Penskoi \cite{PenskoiLawson}] Let $\tau_{m,k}$ be a Lawson surface. Then the induced metric on $\tau_{m,k}$ is an extremal metric 
for the functional $\Lambda_j(M,g)$, where
\begin{equation}
j = 2\left[\frac{\sqrt{m^2+k^2}}{2}\right] + m + k - 1,
\label{j}
\end{equation}
$M = \mathbb{T}^2$ if both $m,k$ are odd and $M=\mathbb{K}$ otherwise.

The corresponding value of the functional 
is 
$$
\Lambda_j(\tau_{m,k}) = 8\pi mE\left(\frac{\sqrt{m^2-k^2}}{m}\right).
$$
\end{proposition}
\begin{proposition}
Let $j$ be defined by formula~(\ref{j}). If $\tau_{m,k}$ is a Lawson torus, then
$$
\Lambda_j(\tau_{m,k}) < 8\pi\left(j-1+\frac{\pi}{\sqrt{3}}\right).
$$
If $\tau_{m,k}$ is a Klein bottle, then 
$$
\Lambda_j(\tau_{m,k}) < 8\pi(j-1)+12\pi E\left(\frac{2\sqrt{2}}{3}\right).
$$
\label{LawsonProp}
\end{proposition}
\begin{proof}
It is sufficient to obtain the inequality
\begin{equation}
\label{LawsonIneq}
j \geqslant mE\left(\frac{\sqrt{m^2-k^2}}{m}\right).
\end{equation}

Let us remark that the function
$$
\varphi(x) = 1 + x - E(\sqrt{1-x^2})
$$
is positive on the interval $[0,1]$. Indeed, 
\begin{equation*}
\begin{split}
E(x) =& \int_0^{\pi/2}\sqrt{1 - x^2\sin^2\psi}\,d\psi\leqslant\int_0^{\pi/2}(\sqrt{1 - \sin^2\psi} + \sqrt{(1-x^2)\sin^2\psi})\,
d\psi\\ =& 1 + \sqrt{1-x^2}.
\end{split}
\end{equation*}

Let us divide both sides of inequality (\ref{LawsonIneq}) by $m$ and denote by $x$ the ratio $\dfrac{k}{m}\in [0,1]$. 
Since 
$$
\left[\sqrt{\frac{m^2+k^2}{2}}\right]\geqslant \left[\frac{m+k}{2}\right]\geqslant\left[\frac{m+1}{2}\right] > \frac{m}{2},
$$
one has that inequality (\ref{LawsonIneq}) follows from the positivity of $\varphi(x)$.
\end{proof} 
\section{Bipolar surfaces to the Lawson surfaces} Let $I\colon N \looparrowright \mathbb{S}^3$ be a minimal immersion. 
\label{BipLawsonEstimate}
A Gauss map $I^*\colon N \to\mathbb{S}^3$ is defined pointwise as
the image of the unit normal in $\mathbb{S}^3$ translated to the origin in $\mathbb{R}^4$. Then the exterior product $\tilde I = I\wedge I^*$ is an immersion  
of $N$ in $\mathbb{S}^5\subset\mathbb{R}^6$. Lawson proved in the paper~\cite{Lawson} that this immersion is minimal. The image $\tilde I(N)$ is 
called a bipolar surface to $N$.

Let us denote by $\tilde \tau_{m,k}$ the bipolar surface to the surface $\tau_{m,k}$. Lapointe proved in the paper~\cite{Lapointe} that
\begin{itemize}
\item if $mk\equiv 0\,\,(\mathrm{mod}\, 2)$ then $\tilde \tau_{m,k}$ is a torus carrying the extremal metric for the 
functional $\Lambda_{4m-2}(\mathbb{T}^2,g)$ and 
$$
\Lambda_{4m-2}(\tilde\tau_{m,k}) = 16\pi mE\left(\dfrac{\sqrt{m^2-k^2}}{m}\right);
$$
\item if $mk\equiv 1\,\,(\mathrm{mod}\, 4)$ then $\tilde \tau_{m,k}$ is a torus carrying the extremal metric for the
functional $\Lambda_{2m-2}(\mathbb{T}^2,g)$ and 
$$
\Lambda_{2m-2}(\tilde\tau_{m,k}) = 8\pi mE\left(\dfrac{\sqrt{m^2-k^2}}{m}\right);
$$
\item if $mk\equiv 3\,\,(\mathrm{mod}\, 4)$ then $\tilde \tau_{m,k}$ is a Klein bottle carrying the extremal metric for the
functional $\Lambda_{m-2}(\mathbb{K},g)$ and 
$$
\Lambda_{m-2}(\tilde\tau_{m,k}) = 4\pi mE\left(\dfrac{\sqrt{m^2-k^2}}{m}\right).
$$
\end{itemize}

\begin{proposition}
If $mk\equiv 1 \,\,(\mathrm{mod}\, 4)$ then the following inequality holds
$$
\Lambda_{2m-2}(\tilde\tau_{m,k})<8\pi\left(2m-3 + \frac{\pi}{\sqrt{3}}\right).
$$
If $mk\equiv 0 \,\,(\mathrm{mod}\, 2)$ then the following inequality holds 
$$
\Lambda_{4m-2}(\tilde\tau_{m,k})<8\pi\left(4m-3 + \frac{\pi}{\sqrt{3}}\right).
$$
If $mk\equiv 3\,\,(\mathrm{mod}\, 4)$ and $\{m,k\}\ne\{3,1\}$ then the following inequality holds
$$
8\pi(m-3) + 12\pi E\left(\frac{2\sqrt{2}}{3}\right) > \Lambda_{m-2}(\tilde\tau_{m,k}).
$$
\label{BipLawsonProp}
\end{proposition}

\begin{proof}
In order to prove the first inequality it is sufficient to prove that
\begin{equation}
\label{dot}
mE\left(\frac{\sqrt{m^2-k^2}}{m}\right)\leqslant (2m-2).
\end{equation}
It is well-known that $E(\tilde k)\leqslant \pi/2$ for $\tilde k\in [0,1]$. This implies that
it is sufficient to prove that 
$$
\pi m\leqslant 4m-4.
$$
This inequality holds for $m\geqslant 5$. The statement for $\tilde \tau_{1,1}$ follows from the fact that $\tilde\tau_{1,1}$ is a Clifford torus and $\Lambda_1(\tau_{1,1}) = 4\pi^2$. 

In the same way, in order to prove the second inequality in Proposition~\ref{BipLawsonProp} it is sufficient to prove that 
$$
\pi m\leqslant 4m-3 + \frac{\pi}{\sqrt{3}}.
$$
This inequality holds for $m\geqslant 2$. 

The third inequality is equivalent the following one 
$$
2(m-3) + 3E\left(\frac{2\sqrt{2}}{3}\right) > m E\left(\frac{\sqrt{m^2-k^2}}{m}\right).
$$
Since $E(\tilde k) < \pi/2$ it is sufficient to prove that
$$
\left(2-\frac{\pi}{2}\right)m > 6 - 3E\left(\frac{2\sqrt{2}}{3}\right).
$$
This inequality holds for $m\geqslant 7$. For the exceptional case $\{m,k\} = \{5,3\}$ one verifies the third inequality explicitly using the tables of elliptic integrals in the book~\cite{Friedman}.  
\end{proof}
\section{Bipolar surfaces to Otsuki tori}
In the paper~\cite{Karpukhin} the following proposition was proved.
\label{BipOtsukiEstimate}
\begin{proposition} The bipolar surface $\tilde O_{p/q}$ to an Otsuki torus $O_{p/q}$ is a torus. 

If $q$ is odd then the metric on bipolar Otsuki torus $\tilde O_{p/q}$ is extremal for the functional $\Lambda_{2q+4p-2}(\mathbb{T}^2,g)$ and $\Lambda_{2q+4p-2}(\tilde O_{p/q}) < 4\sqrt{2}q\pi^2$. 

If $q$ is even then the metric on bipolar Otsuki torus $\tilde O_{p/q}$ is extremal for the functional $\Lambda_{q+2p-2}(\mathbb{T}^2,g)$ and  $\Lambda_{q+2p-2}(\tilde O_{p/q}) < 2\sqrt{2}q\pi^2$.\label{bipOtsuki}
\end{proposition}

\begin{proposition}
If $q$ is even then the following inequality holds,
$$
\Lambda_{q+2p-2}(\tilde O_{p/q}) < 8\pi\left(q+2p-3+\frac{\pi}{\sqrt{3}}\right).
$$
If $q$ is odd then one has the following inequality,
$$
\Lambda_{2q+4p-2}(\tilde O_{p/q}) < 8\pi\left(2q+4p-3+\frac{\pi}{\sqrt{3}}\right).
$$
\label{BipOtsukiProp}
\end{proposition}
\begin{proof}
If $q$ is even, then we have
$$
8\pi\left(q+2p-3+\frac{\pi}{\sqrt{3}}\right)>8\pi(q+2p-2)>12\pi q > 2\sqrt{2}\pi^2q.
$$
We used the inequalities $2p>q$ and $p>1$ in order to prove the last inequality.
In the same way, if $q$ is odd, then we have
$$
8\pi\left(2q+4p-3+\frac{\pi}{\sqrt{3}}\right)>8\pi(2q+4p-2)>24\pi q > 4\sqrt{2}\pi^2q.
$$ 
\end{proof}

Now it is easy to see that Propositions~\ref{OtsukiEstimateProp},~\ref{LawsonProp},~\ref{BipLawsonProp},~\ref{BipOtsukiProp} together with Proposition~\ref{LowerBound} imply Theorem~\ref{MainTheorem}.

\section{Clifford torus} 
Let us represent the Clifford torus as a flat torus with the square lattice with edges equal to $2\pi$. In this case 
\label{CliffordEstimate}
the Laplace-Beltrami coinsides up to a sign with the classical two-dimensional Laplace operator. 
Therefore, using the separation of variables one obtains that the eigenfunctions are 
$$
\sin nx\sin my, \quad \sin nx\cos ly,\quad \cos kx\sin my, \quad\cos kx\cos ly,
$$ 
where $n,m \in \mathbb{N}$ and $k,l\in\mathbb{Z}_{\geqslant 0}$. Then, the eigenvalues are equal to $n^2+m^2$, 
$n^2+l^2$, $k^2+m^2$, $k^2+l^2$ respectively.
\begin{proposition} For the Clifford torus the Weyl's counting function $N(\lambda)$ is equal to the number of integer points
in the open disk of radius $\sqrt{\lambda}$ with the center at the origin of ${R}^2$.
\end{proposition}
\begin{proof} Let us introduce an one-to-one correspondence $\nu$ between eigenfunctions and integer points in
$\mathbb{R}^2$. We set
$$
\left\{
   \begin{array}{lcl}
	\nu(\sin nx\sin my) &=& (n,m),\\
	\nu(\sin nx\cos ly) &=& (n,-l),\\
	\nu(\cos kx\sin my) &=& (-k,m),\\
	\nu(\cos kx\cos ly) &=& (-k,-l).\\
   \end{array}
\right.
$$
   
Let us also remark that the eigenvalue of the function $f$ is equal to the squared distance between $(0,0)$ and $\nu(f)$.
This observation completes the proof.
\end{proof}
\subsection{Proof of Proposition~\ref{CliffordTheorem}.}
It is easy to check that the set of functions 
$$
(\sin kx,\cos kx,\sin ky,\cos ky)
$$ 
form an isometrical immersion of Clifford torus in the unit sphere. The same is true for the set 
$$
(\sin kx\sin ky,\sin kx\cos ky, \cos kx\sin ky, \cos kx\cos ky)
$$ 
and the set
\begin{equation*}
\begin{split}
(\sin k&x\sin ly,\sin kx\cos ly, \cos kx\sin ly, \cos kx\cos ly,\\
&\sin lx\sin ky,\sin lx\cos ky, \cos lx\sin ky, \cos lx\cos ky),
\end{split}
\end{equation*}
where $k\ne l$. Therefore, according to Theorem~\ref{th1}, the metric on the Clifford torus is extremal for the functionals
$\Lambda_{N(r^2)}(\mathbb{T}^2,g)$, where $r^2 = n^2+m^2$ with $n,m\in\mathbb{Z}$, and 
$\Lambda_{N(r^2)}(\mathbb{T}_{Cl}) = 4\pi^2 r^2$. 

Let $B_r$ be a disc of radius $r$. Then one has a simple estimate
$$
N(r^2)\geqslant \mathrm{Area}\left(B_{r-\sqrt{2}/2}\right)
=\pi\left(r-\frac{\sqrt{2}}{2}\right)^2.
$$ 

So, it is sufficient to prove that 
$$
2\left(r-\frac{\sqrt{2}}{2}\right)^2 > r^2
$$
and this inequality holds for $r^2\geqslant 6$. And for $r^2<6$ holds the inequality $8\pi N(r^2) > 4\pi r^2$. This inequality can 
be obtained by the direct enumeration of all possible values of $r^2$. This completes the proof of Proposition~\ref{CliffordTheorem}.
\subsection*{Acknowledgements} The author thanks A.V. Penskoi for the statement of this problem, fruitful discussions and invaluable help in the preparation of the manuscript.

The research of the author was partially supported by Dobrushin Fellowship and by Simons-IUM Fellowship.

\textsc{{Department of Geometry and Topology, Faculty of Mechanics and Mathematics, Moscow State University, Leninskie Gory, GSP-1, 
119991, Moscow, Russia}}

\smallskip

\textit{{and}}

\smallskip

\textsc{{Independent University of Moscow, Bolshoy Vlasyevskiy pereulok 11, 119002, Moscow, Russia}}

\smallskip

\textit{E-mail address:} \texttt{karpukhin@mccme.ru}
\end{document}